\documentclass{article}%
\usepackage{amssymb}
\usepackage{amsfonts}
\usepackage{amsmath}
\usepackage{graphicx}%
\providecommand{\U}[1]{\protect\rule{.1in}{.1in}}
\newtheorem{theorem}{Theorem}
\newtheorem{acknowledgement}[theorem]{Acknowledgement}

\newtheorem{corollary}[theorem]{Corollary}

\newtheorem{definition}[theorem]{Definition}
\newtheorem{example}[theorem]{Example}

\newtheorem{lemma}[theorem]{Lemma}

\newtheorem{proposition}[theorem]{Proposition}
\newtheorem{remark}[theorem]{Remark}

\newenvironment{proof}[1][Proof]{\noindent\textbf{#1.} }{\ \rule{0.5em}{0.5em}}
\begin{document}

\title{On boundaries of geodesically complete CAT(0) spaces}
\author{Conrad Plaut\\Department of Mathematics\\The University of Tennessee\\Knoxville TN 37996\\cplaut@utk.edu}
\maketitle

\begin{abstract}
We give concrete, \textquotedblleft infinitesimal\textquotedblright%
\ conditions for a proper geodesically complete CAT(0) space to have
semistable fundamental group at infinity.

\end{abstract}

\section{Introduction}

For a CAT(0) space $X$ there is a notion of boundary $\partial X$; details are
discussed later in this paper, but see also \cite{BH}. If a group $G$ acts
properly and cocompactly by isometries on $X$ then $G$ is called a CAT(0)
group. If one strengthens the CAT(0) assumption to Gromov hyperbolicity it is
well-known that the boundaries of any two such spaces $X$ on which $G$ acts
must have boundaries that are homeomorphic, so there is a topopogically
well-defined notion of the boundary of a hyperbolic group. It was shown by
Swarup (\cite{SW}) in 1996 that connected boundaries of Gromov Hyperbolic
groups must be Peano continuua (see also \cite{BM}, \cite{Bow}). In contrast,
Croke-Kleiner (\cite{CK}) showed in 2000 that the same group may act properly
and cocompactly on homeomorphic CAT(0) spaces with non-homomeomorphic
boundaries, and a definitive statement on the topological structure of
boundaries of one-ended (i.e. with connected boundary) CAT(0) spaces remains
elusive. On the one hand, arbitrary metric compacta can be realized as
boundaries of CAT(0) spaces (attributed to Gromov with a proof sketched in
\cite{GO}, Proposition 2). But if $X$ is a cocompact, proper CAT(0) space then
there are the following known constraints: According to Swenson (\cite{Sw}),
$\partial X$ must be finite dimensional. According to Geoghagen-Ontandeda
(\cite{GO}) if the dimension of $\partial X$ is $d$ then the $d$-dimensional
\v{C}ech cohomology with integer coefficients is non-trivial. In the same
paper, the authors show that cocompact proper CAT(0) spaces must be
\textquotedblleft almost geodesically complete\textquotedblright\ in a sense
attributed to Michael Mihalik that extends the following notion of
\textit{geodesically complete }(also known as the geodesic extension
property): Every geodesic extends to a geodesic defined for all\textit{\ }%
$\mathbb{R}$.

A natural candidate for a definitive general topological statement about
boundaries of proper cocompact CAT(0) spaces is that they are always
\textquotedblleft pointed $1$-moveable\textquotedblright, a concept from
classical shape theory. The reason for this is that Geoghegan-Swenson
(\cite{G}, Theorem 3.1) showed that a one ended proper CAT(0) space has
semistable fundamental group at infinity if and only if the boundary is
pointed $1$-movable, and it is a long-standing open question whether proper,
cocompact CAT(0) spaces all have semistable fundamental group at infinity (or
simply are \textquotedblleft semistable at infinity\textquotedblright). In the
compact metric case, pointed $1$-movable is equivalent to the notion of
\textquotedblleft weakly chained\textquotedblright\ introduced in \cite{PWC}.
The later is very simple to define, but we do not need definitions of any of
these concepts here; rather we use Theorem 1 from \cite{PWC}, stated as
Theorem \ref{CAT} below, which only involves the following new definition from
\cite{PWC}. Let $(x,y)$ be a pair of distinct points in a metric space $X$.
Then $(x,y)$ is called a \textit{(distance) sink} (\cite{PWC}) if $(x,y)$ is a
local minimum of the distance function. That is, $(x,y)$ is not a sink if and
only if there are points $x^{\prime},y^{\prime}$ arbitrarily close to $x,y$
with $d(x^{\prime},y^{\prime})<d(x,y)$. In a metric space $X$, for any
$x_{0}\in X$ and $r>0$, $\Sigma_{x_{0}}(r)$ denotes the metric sphere
$\{y:d(x_{0},y)=r\}$.

\begin{theorem}
\label{CAT}Let $X$ be a proper, geodesically complete CAT(0) space with
connected boundary and $x_{0}\in X$. Suppose there exist some $K>0$ and a
positive real function $\iota$, called the refining increment, such that for
all sufficiently large $t$,

\begin{enumerate}
\item $\underset{s\rightarrow t^{+}}{\lim}\iota(s)>0$ (in particular if
$\iota$ is lower semicontinuous from the right) and

\item if $d(x,y)<\iota(t)$ and $(x,y)$ is a sink in $\Sigma_{x_{0}}(t)$ then
$x,y$ may be joined by a curve in $X_{t}\cap B(x,K)\cap B(y,K)$.
\end{enumerate}

Then $\partial X\ $is weakly chained (hence $\partial X$ is pointed 1-movable
and $X$ is semistable at infinity).
\end{theorem}

We need the following notations to state our main application--we give more
details later in the paper. The $\pi$-truncated metric of any metric space is
the minimum of $\pi$ and the original metric. For any $x$ in a metric space
$X$, let $S_{x}$ denote the space of directions at $x$ with the angle $\alpha$
as metric, the completion of which is known to be a CAT(1) space when $X$ is
CAT($k$). A \textit{local cone} in a CAT($k$) space $X$ is a closed metric
ball $C=\overline{B(o,\rho_{o})}$ for some $o\in X$ called the \textit{apex,}
such that there is an isometry from $C$ into the $k$-cone $C_{k}(S_{o})$ that
takes $o$ to the apex $0$ of the $k$-cone. The number $\rho_{o}>0$ is called
the \textit{cone radius at }$o$. We say that a geodesic space $X$ is (resp.
\textit{uniformly}) \textit{locally conical if every }$x\in X$ is the apex of
a local cone $\overline{B(x,\rho_{x})}$ (resp. and the cover $\mathcal{C}$ by
the interiors of the local cones has a Lebesgue number $\rho>0$). Such a cover
$\mathcal{C}$ is called a \textit{uniform cone cover} of $X$. It is easy to
check that if $X$ is a locally conical geodesic space then $X$ is uniformly
locally conical if either there is a positive lower bound on the cone radii at
all points or $X$ is cocompact.

\begin{theorem}
\label{fin}If $X$ is a proper, geodesically complete CAT(0) space with
connected boundary and a uniform cone covering $\mathcal{C}$ such that for all
$o\in A_{\mathcal{C}}$, the complement of every $\frac{\pi}{2}$-ball in
$S_{o}$ is connected, then $\partial X$ is weakly chained. This condition on
$S_{o}$ is in particular true when

\begin{enumerate}
\item $S_{o}$ (with the angle metric) is a geodesic space or

\item $S_{o}$ has no cut points and is itself locally conical with all cone
radii at least $\frac{\pi}{2}$.
\end{enumerate}
\end{theorem}

With a minor caveat, the geodesic completeness of $X$ is equivalent to $S_{o}$
being geodesically complete for all $o$ (Lemma \ref{gclem}). In other words,
once $X$ is known to be proper, uniformly locally conical and having connected
boundary, the hypotheses of Theorem \ref{fin} reduce entirely to
\textquotedblleft infinitesimal\textquotedblright\ questions about the space
of directions at each apex.

If $K$ is an $M_{k}$-polyhedral complex, we always assume it has a specific
geometric type or \textquotedblleft shape\textquotedblright\ assigned to each
cell, and the set Shapes($K$) of these shapes is finite. Bridson showed in his
thesis (exposition in \cite{BH}) that when Shapes($K$) is finite, $K$ has a
natural geodesic metric induced by this choice of shapes, and we will always
take this metric on $K$. The space of directions at any point is a spherical
($k=1$) polyhedral complex called the link $Lk(x,K)$. In fact, the angle
metric on $Lk(x,K)$ is precisely the $\pi$-truncated metric of the induced
geodesic metric when $Lk(x,K)$ is considered as a spherical polyhedral
complex. Each vertex $v$ of $K$ is the apex of a local cone $\overline
{B(v,\rho)}$, where $\rho$ is at least the infimum of distances to any face
not containing $v$ of a cell that does contain $v$. Every $x\in K$ is
contained in a ball $B(x,\varepsilon)$ that is isometric to a ball of the same
radius in some $\overline{B(v,\rho)}$ with $v$ a vertex such that
$\overline{B(x,\varepsilon)}$ is a local cone, and this $\varepsilon$ has a
positive uniform lower bound (Theorem I.7.39, Lemma I.7.54, \cite{BH}). That
is, $K$ is uniformly locally conical. Now if $x\in B(v,\rho)$ where $v\neq x$,
then $Lk(x,K)$ is isometric to the spherical suspension of $Lk(u,Lk(v,K))$,
where $u$ is the direction of the geodesic from $v$ to $x$ (Lemma
\ref{vectorlem}). Now suppose that $Lk(v,K)$ has at least two points and no
free faces. By Lemma \ref{gclem}, $Lk(v,K)$ is geodesically complete, and
since it is locally conical, $Lk(u,Lk(v,K))$ is also geodesically complete by
Lemma \ref{comp}. Moreover, by Lemma \ref{comp} and Lemma \ref{vectorlem}, the
conditions of Theorem \ref{fin} are always satisfied for any non-vertex $x$ if
they are satsified for every vertex. Putting all of this together we obtain:

\begin{theorem}
\label{Main}Let $K$ be a CAT(0) Euclidean polyhedral complex with Shapes($K$)
finite and connected boundary. If for each vertex $v$ in $K$, $Lk(v,K)$ has at
least two points and no free faces, and the complements of all $\frac{\pi}{2}%
$-balls in $Lk(v,K)$ are connected, then $\partial K$ is weakly chained. In
particular this is true when $Lk(v,K)$ satisfies condition (1) or (2) in
Theorem \ref{fin} at each vertex $v$.
\end{theorem}

\begin{example}
The well-known torus complexes of Croke-Kleiner (\cite{CK}) are locally
isometric to two or four Euclidean half-spaces glued along a line, so the link
at every vertex consists geometrically of semicircles (length $\pi$) attached
\textquotedblleft in parallel\textquotedblright\ to one another at their
endpoints. In particular, the links are geodesic spaces with the angle metric
and have no free faces. Theorem \ref{Main} now provides an easy proof that the
fundamental groups of torus complexes are semistable at infinity. This is a
known result--we do not have a reference but Michael Mihalik explained that
this can be read off from a presentation of the group using some of his
earlier results on the subject. Indeed many of the results of this sort are
based on assumptions about how the groups are presented, for example for
Coxeter and Artin groups (\cite{Mi2}). In contrast, Theorem \ref{Main}
requires no knowledge of a group acting on $K$ and in fact $X$ is not even
required to be cocompact.
\end{example}

We say that a metrized CAT(0) polyhedral complex $K$ with Shapes($K$) finite
satisfies Moussong's condition if for every vertex $v$, every edge in
$Lk(v,K)$ has length at least $\frac{\pi}{2}$. Spherical geometry then implies
that $Lk(v,K)$ is locally conical at every vertex with cone radius at least
$\frac{\pi}{2}$. In this case, Theorem \ref{fin} gives a purely combinatorial
sufficient condition on the link at each vertex:

\begin{corollary}
\label{ss}Let $K$ be a CAT(0) Euclidean polyhedral complex with Shapes($K$)
finite and connected boundary that satisfies Moussong's condition. If the link
at each vertex has at least two points and no free faces or cut points, then
$\partial K$ is weakly chained.
\end{corollary}

\begin{example}
Moussong showed (\cite{Mou}) that the Davis complex (\cite{Da88}) of any
Coxeter group has a non-positively curved metric statisfying Moussong's
condition. See also \cite{DCG} for background on Coxeter groups. Mihalik
showed in 1996 (\cite{Mi2}) that all Coxeter groups are semistable at
infinity, but Corollary \ref{ss} gives an alternative proof when the link has
no free faces or cut points. However, since any polyhedron may be the link of
the Davis complex of a Coxeter group (see Lemma 7.2.2, \cite{DCG}), Corollary
\ref{ss} does not apply in this way to all Coxeter groups.
\end{example}

\begin{example}
[Geoghagen]Take a unit square. Attach 16 unit squares around the boundary
wrapping around twice (topologically attaching a Moebius band via its median
circle to the boundary of the square). Add 32 unit squares wrapping twice
around the new boundary. Continue this process to create an infinite square
complex. Up to isometry there are only two kinds of vertices:
\textquotedblleft corner vertices\textquotedblright, which lie in seven
squares and \textquotedblleft side vertices\textquotedblright, which lie in
six squares. Therefore this complex is uniformly locally conical. The link at
any side vertex is isometric to three semicircles glued at their endpoints
with the induced geodesic metric, which has no free faces and the complements
of $\frac{\pi}{2}$-balls are connected. The link at the corner vertices is
topologically the same, but geometrically consists of two segments of length
$\frac{3\pi}{2}$ and one segment of length $\frac{\pi}{2}$. If $u$ is the
point at the center of the latter segment then the complement of
$B(u,\frac{\pi}{2})$ has two components. Note also that the induced geodesic
metric and the angle metric on the link do not coincide at corner vertices;
the angle metric is not geodesic. Moreover, Moussong's condition is not
satisfied, although none of the links has a cut point. The boundary of this
space is the 2-adic solenoid, which is known to topologists to not be pointed
1-connected, and was alternatively shown in \cite{PWC} to not be weakly
chained. This shows that the condition about $\frac{\pi}{2}$-balls in Theorem
\ref{fin} cannot simply be removed. Also note that the failure of Theorem
\ref{fin} in this case is strictly for geometric reasons, not combinatorial or
topological ones. Finally, note that although there are only finitely many
isometry types of local cones in this example, it is not cocompact. For
example, only the first square has four \textquotedblleft corner
vertices\textquotedblright.
\end{example}

\section{Cones and suspensions}

We recall a couple of special cases of Berestovskii's consruction of metric
cones and suspensions and establish some basic results for which we have no
references. Let $S$ be a metric space with distance $\rho$ and $\pi$-truncated
metric denoted by $\alpha$. The Euclidean cone $C_{0}(S)$ consists of
$[0,\infty)\times S$ with all points of the form $(0,v)$ identified to a
single point called the apex. We will denote the equivalence class of the
point $(t,v)$ by $tv$, and the apex will be denoted by $0$ or $0v$ depending
on the situation. Note that with this notation $S$ is naturally identified
with the set of all all $1v$ in $X$, which we will denote by $1S$. However
this identification is not generally an isometry and therefore we will
distinguish notationally between elements $1v$ in $X$ and $v$ in $S$. If
$u(s)$ is a curve in $S$ then the curve $tu(s)$ in $X$ will be denoted simply
by $tu$. $X$ is metrized analogously to how $\mathbb{R}^{2}$ is metrized as
the cone of the unit circle with angle as metric. That is, for $s,t>0$ and
$v,w\in S$, $d(tv,sw)^{2}=s^{2}+t^{2}-2st\cos\alpha(v,w)$ (cones $C_{k}X$ for
other curvatures $k$ use the corresponding cosine laws).

Let $M_{k}^{2}$ denote the 2-dimensional space form of constant curvature $k$.
We are primarily interested in $k=0,1$, so $M_{0}^{2}$ is the plane and
$M_{1}^{2}$ is the sphere of curvaure $1$. The former has diameter $\infty$
and the latter has (intrinsic!) diameter $\pi$. Recall that a CAT($k$) space
$X$ is a metric space such that if $d(x,y)$ is less than the diameter of
$M_{k}^{2}$ then $x,y$ are joined by a geodesic and if a geodesic triangle has
a comparison triangle in $M_{k}^{2}$ (no restriction for $k\leq0$) then
Alexandrov's comparisons for curvature $\leq k$ hold (see \cite{BH} for more
details). Berestovskii (\cite{Ber83}) showed in 1983 that $S$ is a CAT(1)
space if and only if $C_{0}S$ is a CAT(0) space (and more strongly the same is
true for $k$-cones for any $k\in\mathbb{R}$). We assume now that $S$ is a
CAT(1) space and review a few facts about geodesics in $C_{0}S$. Suppose that
$u:[0,K]\rightarrow S$ is an arclength parameterized geodesic from $v$ to $w$
in $S$ of length $K<\pi$. For simplicity we consider the constant map to be a
geodesic from $v$ to $v$. By definition of the metric, the function
$f_{\alpha}(s,t)=su(t)$ is an isometry from a Euclidean sector $E(v,w)$ of
angle $K$ in the plane parameterized with polar coordinates, to the set
$Z(v,w)=\{su(t):0\leq s\leq\infty,0\leq t\leq K\}$. Therefore the curves in
$Z(v,w)$ corresponding to line segments in the Euclidean sector are geodesics
in $Z$.

A geodesic in the case $v=w$ is called a \textit{radial geodesic}, i.e of the
form $\gamma_{w}(t)=tw$, $0\leq t<\infty$. Now suppose that $\gamma$ is a
geodesic from $tv$ to $sw$ that does not meet the apex. Since geodesics are
unique, the concatenation of the radial geodesics between the apex and $tv$
and $sw$ cannot be a geodesic and we conclude that $\alpha(v,w)<\pi$. This
means that $\gamma$ corresponds to a line in $E(v,w)$.

When $\alpha(v,w)=\pi$, for any $0<r_{1},r_{2}$, by definition $d(r_{1}%
v,r_{2}w)=\left\vert r_{1}-r_{2}\right\vert $ and therefore the concatenation
of the radial geodesic from $r_{1}v$ to $0$ with the radial geodesic from $0$
to $r_{2}w$ is a geodesic from $r_{1}v$ to $r_{2}w$. All of these geodesics in
$X$ are unique since $X$ is a CAT(0) space.

\begin{example}
\label{cwc}The Euclidean cone of an arbitrary metric space $S$ is
\textquotedblleft sink-free\textquotedblright\ in the sense that it has no
sinks (\cite{PWC}). In fact one can move any pair of points towards the apex,
strictly decreasing the distance between them (and any non-apex point may
similarly move towards the apex). Such cones need not be locally path
connected (e.g. when $S$ is a Cantor set).
\end{example}

The spherical suspension $\Sigma S$ of a metric space $S$ is defined
analogously using the $\pi$-truncated metric, taking the product of the space
with $[0,\pi]$, identifying each $0\times S$ and $\pi\times S$ with points
$\overline{0}$ and $\overline{\pi}$, respectively. The space is metrized using
the spherical cosine law, i.e. as $S^{2}$ as metrized as the suspension of a
circle of length $2\pi$. For $u\in S$ and $\theta\in(0,\pi)$ we denote the
point corresponding to the ordered pair $(u,\theta)$ by $u_{\theta}$.

\begin{lemma}
\label{vectorlem}If $S$ is a CAT(1) space then the spaces of directions at a
point $t_{0}v\in C_{0}S$ is isometric to $S$ when $t_{0}=0$. If $t_{0}>0$ then
$S_{t_{0}v}$ is isometric to the spherical suspension of $S_{v}$, and if
$B(v,\rho)$ is a local cone in $S$ then $B(t_{0}v,\rho)$ is a local cone in
$C_{0}S$.
\end{lemma}

\begin{proof}
[Sketch of Proof]The case when $t_{0}=0$ simply follows from the definition of
the cone metric. Suppose $t_{0}>0$ and let $\Gamma_{0}$ denote the segment of
the radial geodesic of $v$ outward from $t_{0}v$. Suppose that $\gamma$ is a
geodesic in $S$ starting at $v$. By definition of the cone metric, every
geodesic in $C_{0}S$ starting at $t_{0}v$ in the same sector as $\gamma$ is
uniquely determined by the angle $\theta$ between it and $\Gamma_{0}$; we will
denote any such geodesic by $\gamma_{\theta}$. This identifies $S_{t_{0}v}$
(as a set) with the spherical suspension of $S_{v}$, with $\overline{0}$
corresponding to the direction of $\Gamma_{0}$.

Suppose that $\gamma_{\theta_{1}}^{1}$ and $\gamma_{\theta_{2}}^{2}$ are
geodesics starting at $t_{0}v$, with directions $v_{1},v_{2}$, respectively at
$t_{0}v$. The cases when $\theta_{i}$ equal $0$ or $\pi$ are trivial; suppose
$0<\theta_{1},\theta_{2}<\pi$. Let $\overline{w_{i}}$ be unit vectors in the
$(x,y)$ plane with $\left\langle \overline{w_{1}},\overline{w_{2}%
}\right\rangle =\cos\angle(\overline{w_{1}},\overline{w_{2}})=\cos
\angle\left(  \gamma^{1},\gamma^{2}\right)  $. Let $\overline{v_{i}}$ be unit
vectors whose orthogonal projections onto the $(x,y)$-plane are parallel to
$\overline{w_{i}}$ and $\alpha(\overline{v_{i}},\overline{w_{i}})=\theta_{i}$.
Writing $\overline{v_{i}}=\overline{w_{i}}+(\overline{v_{i}}-\overline{w_{i}%
})$ and cancelling orthogonal terms we have%
\begin{equation}
\left\langle \overline{v_{1}},\overline{v_{2}}\right\rangle =\left\langle
\overline{w_{1}},\overline{w_{2}}\right\rangle +\left\langle \overline{v_{1}%
}-\overline{w_{1}},\overline{v_{2}}-\overline{w_{2}}\right\rangle
=\left\langle \overline{w_{1}},\overline{w_{2}}\right\rangle \pm\left\Vert
\overline{v_{1}}-\overline{w_{1}}\right\Vert \left\Vert \overline{v_{2}%
}-\overline{w_{2}}\right\Vert \label{vector}%
\end{equation}
with the sign of the last term depending on whether $\theta_{i}$ are on the
same side of $\frac{\pi}{2}$ (in which case it is $+$). This shows that
$\cos\angle\left(  \overline{v_{1}},\overline{v_{2}}\right)  $, which is what
we need to show is equal to $\cos\alpha(\gamma_{\theta_{1}},\gamma_{\theta
_{2}})$ may be calculated using only $\left\langle \overline{w_{1}}%
,\overline{w_{2}}\right\rangle $ and lengths of vectors in the $z$-direction.
Now
\[
\cos\angle(\gamma_{\theta_{1}}^{1},\gamma_{\theta_{2}}^{2}%
)=\underset{t\rightarrow0}{\lim}1-\frac{d(\gamma_{\theta_{1}}^{1}%
(t),\gamma_{\theta_{2}}^{2}(t))^{2}}{2t^{2}}\text{.}%
\]
By definition of the cone metric, the right term may be computed from
$d(\gamma^{1}(t),\gamma^{2}(t))$ using the Euclidean Formula \ref{vector},
and
\[
\underset{t\rightarrow0}{\lim}1-\frac{d(\gamma^{1}(t),\gamma^{2}(t))^{2}%
}{2t^{2}}=\cos\angle(\gamma^{1},\gamma^{2})=\left\langle \overline{w_{1}%
},\overline{w_{2}}\right\rangle \text{.}%
\]
This shows the first part of the lemma. For the second part note that if
$B(v,\rho)$ is a local cone in $S$ then for all $0<t\leq\rho$,
\[
d(\gamma^{1}(t),\gamma^{2}(t))^{2}=2t^{2}(1-\cos\angle(\gamma^{1},\gamma
^{2})\text{.}%
\]
Therefore the above limit is constant and we see that
\[
d(\gamma_{\theta_{1}}^{1}(t),\gamma_{\theta_{2}}^{2}(t))^{2}=2t^{2}%
(1-\cos\angle(\gamma_{\theta_{1}}^{1},\gamma_{\theta_{2}}^{2})\text{,}%
\]
i.e. $B(t_{0}v,\rho)$ is a local cone in $C_{0}S$.
\end{proof}

\begin{lemma}
\label{DNB}(Radial Geodesics Don't Bifurcate) Let $S$ be a CAT(1) space. If
$\gamma$ is a geodesic in $C_{0}S$ that intersects a radial geodesic $\beta$
in more than one point then $\gamma$ is a (possibly infinite) segment of
$\beta$.
\end{lemma}

\begin{proof}
Let $\beta(t)=tu$ for some $u\in S$. By assumption, $\gamma$ contains two
points $t_{1}u,t_{2}u$ with $0\leq t_{1}<t_{2}$. By uniqueness, $\gamma=\beta$
between those two points. Now suppose that another point $sw$ lies on $\gamma$
with $w\neq u$ and $s>t_{2}$. Then the segment of $\gamma$ from $t_{1}u$ to
$sw$ lies in $Z(u,w)$ and the segment of $\gamma$ from $t_{1}u$ to $t_{2}u$
also lies in $Z(u,w)$. But his means that one segment of $\gamma$ in $Z(u,w)$
is radial and another segment is not, which is impossible in Euclidean
geometry. The proof for $s<t_{1}$ is similar, showing that $\gamma(t)=tu$
where ever it is defined.
\end{proof}

\begin{remark}
\label{radial}One known consequence of the above lemma is that there is the
continuous \textquotedblleft radial retraction\textquotedblright\ from any
$B(0,\rho)\backslash\{0\}$ onto the sphere $\Sigma_{0}(\rho)$, which just
takes every $tv$ with $0<t\leq\rho$ to $\rho\cdot v$.
\end{remark}

\begin{definition}
Let $X$ be a metric space. A curve $c$ in $X$ is called a local geodesic if
the restriction of $c$ to any sufficiently small closed interval is a geodesic
(so in a Riemannian manifold this would simply be what is normally referred to
as a \textquotedblleft geodesic\textquotedblright). $X$ is called geodesically
complete if for every non-constant geodesic $\gamma:[a,b]\rightarrow X$ there
is some $\varepsilon>0$ such that $\gamma$ extends to a curve $\gamma
^{e}:[a,b+\varepsilon]\rightarrow X$ such that the restriction of $\gamma^{e}$
to $[b-\varepsilon,b+\varepsilon]$ is a geodesic.
\end{definition}

Note that if $X$ is a complete metric space then $X$ is geodesically complete
if and only if every local geodesic extends to a local geodesic defined on all
of $\mathbb{R}$.

\begin{lemma}
\label{gclem}If $S$ is a complete CAT(1) space with at least two points then
the following are equivalent:

\begin{enumerate}
\item $X=C_{0}(S)$ is geodesically complete.

\item $S$ is geodesically complete.

\item Every non-trivial geodesic in $S$ extends to a geodesic of length at
least $\pi$.
\end{enumerate}
\end{lemma}

\begin{proof}
Since there is a 1-1 correspondence between non-radial geodesics in $X$ and
geodesics in $S$, if $X$ is geodesically complete then so is $S$. Suppose $S$
is geodesically complete, $u\in S$ and $\gamma$ is a non-trivial geodesic in
$S$ starting at $u$. Then $\gamma$ extends as a local geodesic of length $\pi$
to a point $v$. We claim that $d(u,v)=\pi$, which means that this extension is
in fact a geodesic, completing the proof of the third part. We assume that
$\gamma:[0,\pi]\rightarrow X$ is parameterized by arclength. Consider the
following statement $S(t):d(u,\gamma(t))=t$. Since $\gamma$ is a local
geodesic, $S(t)$ is true for small positive $t$. By continuity of the distance
function, if $S(s)$ is true for all $s\leq t$ then $S(t)$ is true. Therefore
we need only show that if $S(t)$ is true for some $t<\pi$ then
$S(t+\varepsilon)$ is true for some $\varepsilon>0$. Since $\gamma$ is a local
geodesic, there is some $\varepsilon>0$ such that the restrictions of $\gamma$
to $[t,t+\varepsilon]$ and $[t-\varepsilon,t+\varepsilon]$ are geodesics. The
first of these statements implies that the angle between the reversal of
$\gamma$ starting at $\gamma(t)$ and the restriction of $\gamma$ starting at
$\gamma(t)$ is $\pi$. Since both segments are \textit{bona fide }geodesics we
may apply the CAT(1) condition to conclude that $d(u,\gamma(t+\varepsilon
))=t+\varepsilon$, completing the proof of $2\Rightarrow3$.

If the third part is true then by the 1-1 correspondence mentioned above,
every non-radial geodesic in $X$ extends to a geodesic defined on $\mathbb{R}%
$, and all radial geodesics by definition extend outwards from the apex. The
only remaining question is whether the reversal any radial geodesic
$\gamma_{u}$ extends through the apex. Since $S$ has at least two points there
is some $v\neq u$ in $S$. If $\alpha(u,v)=\pi$ then $\angle(\gamma_{u}%
,\gamma_{v})=\pi$ and $\gamma_{v}$ extends $\overline{\gamma_{v}}$ as a
geodesic beyond the apex. If $\alpha(u,v)<\pi$ then $u$ and $v$ are joined by
a geodesic in $S$, which extends to a geodesic to some $v$ such that
$\alpha(u,v)=\pi$, completing the proof.
\end{proof}

\begin{corollary}
\label{comp}If $S$ is a geodesically complete CAT(1) space then the spherical
suspension $\Sigma S$ of $S$ is geodesically complete and the complement of
every $\frac{\pi}{2}$-ball in $\Sigma S$ is (path) connected.
\end{corollary}

\begin{proof}
That $\Sigma S$ is geodesically complete follows from Lemma \ref{gclem} and
the definition of the spherical suspension metric. Let $v_{\theta}\in\Sigma
S$; without loss of generality suppose $\theta\leq\frac{\pi}{2}$. We will show
that if $w_{\mu}\neq\overline{\pi}$ with $d(w_{\mu},v_{\theta})\geq\frac{\pi
}{2}$ then there is a path from $w_{\mu}$ to $\overline{\pi}$ that stays
outside $B(v_{\theta},\frac{\pi}{2})$. If $v=w$ then by definition of the
spherical suspension metric there is a uniquely determined isometric circle
determined by $v=w$ and we may move in either direction from $w_{\mu}$
(depending on whether $\theta>\mu$) to $\overline{\pi}$, staying outside
$B(v_{\theta},\frac{\pi}{2})$. If $w\neq v$ then the geodesic from $v$ to $w$
extends to length $\pi$. Therefore we may move from $w_{\mu}$ away from
$v_{\theta}$ along the corresponding geodesic to the \textquotedblleft
antipodal point\textquotedblright\ of $v_{\theta}$ and proceed as in the first step.
\end{proof}

\section{M$_{k}$-polyhedral complexes}

Recall that a \textit{free face} in an $M_{k}$-polyhedral complex is a face
that lies in exactly one cell of higher dimension. The proof of the next lemma
involves Proposition II.5.10, \cite{BH} (and uses some similar arguments),
which states that an $M_{k}$-polyhedral complex with curvature bounded above
and finite shapes is geodesically complete if and only if it has no free
faces. This statement is not quite correct according to the traditional
definitions because discrete complexes are geodesically complete (there are no
non-trivial local geodesics), and even if one considers the empty set as a
face of dimension $-1$, strictly speaking it is free if and only if the
complex consists of exactly one vertex. And for example discrete spherical
complexes of curvature $\leq1$ occur as the space of directions in
1-dimensional Euclidean complexes. The proof of Proposition II.5.10 is by
induction on dimension starting with $n=0$, which is precisely when the
statement is not true. But this minor issue is easily solved by starting with
$n=1$ and handling the discrete case as a special case.

\begin{lemma}
\label{FF}If $K$ is an $M_{k}$-polyhedral complex with non-positive curvature
and Shapes($K$) finite then $K$ is geodesically complete if and only if the
link at each vertex has at least two vertices and is either discrete or has no
free faces.
\end{lemma}

\begin{proof}
Since $K$ has non-positive curvature, the space of directions at each point,
hence the link at each vertex $v$ is a CAT(1) space. Suppose $X$ is not
geodesically complete. Since $X$ has non-positive curvature it is connected
and hence this is equivalent to having a free face $F$. If $F$ is a vertex
then at that vertex the link is a single point. If is $F$ higher dimensional
then let $v$ be any vertex of $F$ and $E_{1}$ be an edge containing $v$, not
contained $F$ but contained in the unique higher dimensional cell containing
$F$. Let $E_{2}$ be an edge in $F$ containing $v$. Then the geodesic in
$Lk(v,K)$ from the directions $u_{1}$, $u_{2}$ corresponding to $E_{1}$,
$E_{2}$, respectively, cannot be extended beyond $u_{2}$. That is, $Lk(v,K)$
is not geodesically complete, and since it is not discrete it has a free face.

Conversely, if $Lk(v,K)$ is a single vertex for some vertex $v$ in $K$ then
there is an edge in that direction. Moreover, there is no edge having angle
$\pi$ with that edge, so the edge cannot be extended as a geodesic past $v$.
That is, $K$ is not geodesically complete. Finally, suppose there is some
vertex $v$ such that $Lk(v,K)$ is not discrete and has a free face. Since
$Lk(v,K)$ is not discrete, it is not geodesically complete. But some
$B(v,\varepsilon)$ is isometric to $B(0,\varepsilon)$ in $C_{0}Lk(v,K)$ and
the proof is finished by Lemma \ref{gclem}.
\end{proof}

\begin{example}
\label{cat}Let $S$ be a complete $\pi$-geodesic space of diameter $\pi$ (e.g.
a CAT(1) space with $\pi$-truncated metric $\alpha$). Then $S$ is a length
space if and only if $S$ is sink-free. Necessity follows from Example 28 in
\cite{PWC}. For the converse, we need only consider the case $d(x,y)=\pi$. If
$(x,y)$ is not a sink, there exist points $x^{\prime},y^{\prime}$ arbitrarily
close to $x,y$ such that $d(x^{\prime},y^{\prime})<\pi$ and since $\alpha$ is
a $\pi$-geodesic metric, $x^{\prime},y^{\prime}$ are joined by a geodesic.
Then a midpoint between $x^{\prime},y^{\prime}$ is an \textquotedblleft almost
midpoint\textquotedblright\ for $x,y$, which is classically known to be
sufficient to show that $X$ is a length space (cf. \cite{Ph}, Proposition 7
for an exposition). In particular, if $S$ is compact then $S$ is geodesic if
and only if it is sink-free.
\end{example}

\begin{proposition}
\label{largeprop}Suppose $X$ is a geodesically complete CAT(0) space,
$x_{0}\in X$, $r>0$, and $(x,y)$ is a sink in $\Sigma_{x_{0}}(r)$. In
addition, suppose that $x,y$ lie in a local cone with apex $o$ and cone radius
$\rho>0$. Then

\begin{enumerate}
\item the geodesics $\gamma_{x},\gamma_{y}$ coincide up to $o$,

\item $d(x,o)=d(y,o)=\frac{d(x,y)}{2}$ and

\item $\rho\geq\frac{d(x,y)}{2}$.
\end{enumerate}
\end{proposition}

\begin{proof}
Let $\delta:=d(x,y)$, $\beta=\gamma_{xy}$ and $\gamma^{x},\gamma^{y}$ be
geodesics starting at $x,y$, respectively, that extend $\gamma_{x},\gamma_{y}$
to geodesic rays. Suppose first that $\angle(\gamma^{x},\beta)<\pi$. Since
$S_{x}$ is $\pi$-geodesic, there is a geodesic $\xi$ starting at $x$ such that
$\angle(\gamma_{x},\xi),\angle(\xi,\beta)<\frac{\pi}{2}$. By the triangle
inequality in $S_{x}$, $\angle(\gamma_{x},\xi)>\frac{\pi}{2}$ and so by the
CAT(0) condition all points on $\xi$ lie strictly outside $\Sigma_{x_{0}}(r)$.
Since $\angle\left(  \xi,\beta\right)  <\frac{\pi}{2}$, by the
\textquotedblleft single-sided limit\textquotedblright\ method to measure
angles (c.f. Proposition 3.5 in \cite{BH}), any point $x^{\prime}$
sufficiently close to $x$ on $\xi$ satisfies $d(x^{\prime},y)<d(x,y)$. In the
plane, consider the comparison triangle with corners $X_{0},X^{\prime},Y$
corresponding to the one determined by $x_{0},x^{\prime},y$. Since
\[
d(X_{0},X^{\prime})=d(x_{0},x^{\prime})>r=d(x_{0},y)=d(X_{0},Y)\text{,}%
\]
by elementary geometry, if $Z$ is the point on the segment $X_{0}X^{\prime}$
with $d(X_{0},Z)=r$, $d(Z,Y)<d(X^{\prime},Y)$. By the CAT(0) condition, if $z$
is the projection of $x^{\prime}$ onto $\Sigma_{x_{0}}(r)$ then $d(z,y)\leq
d(Z,Y)<d(X^{\prime},Y)=d(x^{\prime},y)$. That is, $z$ is arbitrarily close to
$x$ but $d(z,y)<d(x,y)$. By definition, $(x,y)$ is not a sink in
$\Sigma_{x_{0}}(r)$, a contradiction.

Therefore we may assume that
\begin{equation}
\angle(\gamma^{x},\beta)=\angle(\gamma^{y},\overline{\beta})=\pi\label{pi}%
\end{equation}
and by Equation \ref{pi} $\gamma:=\overline{\gamma^{x}}\ast\beta\ast\gamma
^{y}$ is a geodesic. Assume first that $o$ does not lie on $\gamma$ and
consider the Euclidean sector $E$ determined by $\gamma$ in the local cone. At
$x$ and $y$ let $\beta_{x}$ and $\beta_{y}$ be geodesics corresponding to
lines in $E$ that are perpendicular the line corresponding to $\gamma$ with
the same orientation. That is, moving an arbitrarily small but equal amount
along $\beta_{x}$ and $\beta_{y}$ to points $x^{\prime},y^{\prime}$ we have
$d(x^{\prime},y^{\prime})=d(x,y)$. Since $\angle\left(  \gamma^{x}%
,\overline{\gamma_{x}}\right)  =\pi$ and $\angle\left(  \gamma^{x},\beta
_{x}\right)  =\frac{\pi}{2}$, by the triangle inequality $\angle\left(
\beta_{x},\overline{\gamma_{x}}\right)  \geq\frac{\pi}{2}$, and similarly,
$\angle\left(  \beta_{y},\overline{\gamma_{y}}\right)  \geq\frac{\pi}{2}$. By
the CAT(0) condition, $d(x_{0},x^{\prime}),d(x_{0},y^{\prime})>r$. Projecting
$x^{\prime},y^{\prime}$ onto $\Sigma_{x_{0}}(r)$ strictly reduces
$d(x^{\prime},y^{\prime})=d(x,y)$, showing $(x,y)$ is not a sink in
$\Sigma_{x_{0}}(r)$, a contradiction.

Therefore $o$ must lie on $\gamma$, hence on $\gamma^{x}$, $\beta$, or
$\gamma^{y}$. Suppose $o$ lies on $\gamma^{x}$ and $o\neq x$. Then
$\gamma_{ox}\ast\beta$ is radial and $\gamma_{ox}\ast\overline{\gamma_{x}}$ is
radial as long as it lies in the local cone at $o$. Since $y$ is also in the
local cone, $\gamma_{ox}\ast\overline{\gamma_{x}}$ lies in the local cone for
at least length $d(x,y)$ beyond $x$. By the triangle inequality, the midpoint
$m$ of $\beta$ satisfies $d(x_{0},m)\geq r-\frac{d(x,y)}{2}$. Since
$\overline{\gamma_{x}}$ is a geodesic to $x_{0}$ and $d(x_{0},y)=r$,
$\overline{\gamma_{x}}$ and $\beta$ cannot coincide for the entire length of
$\beta$. By definition this means that the radial geodesic $\gamma_{ox}%
\ast\beta$ bifurcates inside the local cone, a contradiction to Lemma
\ref{DNB}.

Supose that $o=x$. Then $\beta\ast\gamma^{y}$ is a radial geodesic that
coincides with $\gamma_{y}\ast\gamma^{y}$ on $\gamma^{y}$. But the former does
not pass through $x$ and hence the radial geodesic $\gamma_{y}\ast\gamma^{y}$
bifurcates somewhere along $\beta$, which is in the local cone, a
contradiction. Similarly $o$ does not lie on $\gamma^{y}$.

Next suppose that $o$ lies on $\beta$, so $\gamma_{ox}$ is radial. By a
similar argument to what we have used above, the fact that radial geodesics do
not bifurcate implies that $\gamma_{x}$ must coincide with $\gamma_{ox}$ from
$o$ to $x$. By symmetry, $\gamma_{y}$ must coincide with $\gamma_{oy}$ from
$o$ to $y$ and $o=m$. Since $\beta$ is a geodesic, $d(x,o)=d(y,o)=\frac
{d(x,y)}{2}$. Since $x$ and $y$ lie in the local cone and $o$ is the apex,
$\rho\geq\frac{d(x,y)}{2}$.
\end{proof}

The next lemma is probably known but we do not have a reference. It is useful
because it allows us to replace a given cover by metric balls with a cover by
metric balls of uniformly large size.

\begin{lemma}
\label{cover}Let $\mathcal{C=\{}B(x_{\alpha},r_{\alpha})\}_{\alpha\in\Lambda}$
be a covering of a geodesic space $X$ by metric balls. If $\mathcal{U}$ has a
Lebesgue number $\lambda>0$ then there is a subcovering $\mathcal{C}^{\prime
}\mathcal{=\{}B(x_{\alpha},r_{\alpha})\}_{\alpha\in\Lambda^{\prime}}$ of $X$
with the following properties:

\begin{enumerate}
\item For all $\alpha\in\Lambda^{\prime}$, $r_{\alpha}\geq\frac{\lambda}{4}$
also covers $X$.

\item $\mathcal{C}^{\prime}$ has Lebesgue number $\frac{\lambda}{2}$.
\end{enumerate}
\end{lemma}

\begin{proof}
Without loss of generality, if $X$ has finite diameter then we can assume that
$\lambda$ is less than the diameter $D$ of $X$; if $X$ is unbounded let
$D>\lambda$ be arbitrary. In either case, since $X$ is a geodesic space, there
exist $a,b$ such that $d(a,b)=D$. Then for any $x\in X$ $d(x,a)\geq\frac{D}%
{2}$ or $d(x,b)\geq\frac{D}{2}$. Moving along the geodesic $\gamma_{xa}$ or
$\gamma_{xb}$ there is some point $y\in X$ such that $d(x,y)=\frac{\lambda}%
{2}<\lambda$. Therefore there is some $B(x_{\alpha},r_{\alpha})$ containing
both $x$ and $y$. But since $d(x,y)=\frac{\lambda}{2}$, by the triangle
inequality, $\frac{\lambda}{2}\leq2r_{\alpha}$, showing that the balls of
radius $\geq\frac{\lambda}{4}$ in $\mathcal{C}$ cover $X$. Now suppose that
$d(x,y)<\frac{\lambda}{2}$. If $z$ is the midpoint of a geodesic from $x$ to
$y$, by what we have just shown there is some $B(x,r_{\alpha})$ containing $z$
with $r_{\alpha}\geq\frac{\lambda}{4}$, and by this ball must contain both $x$
and $y$.
\end{proof}

\begin{proof}
[Proof of Theorem \ref{fin}]Let $\rho$ be a Lebesgue number for a uniform cone
cover of $X$. By Lemma \ref{cover} we can assume that the cone radii for the
local cones are all at least $\frac{\rho}{2}$. For every $t>0$ let
\[
\iota(t):=2\left(  t-\sqrt{t^{2}-\frac{\rho^{2}}{4}}\right)  >0
\]
and assume that $t$ is large enough that $\iota(t)<\rho$. Since $\iota$ is
positive and continuous we need only verify Theorem \ref{CAT}.2 for any $t$.
Suppose that $d(x,y)<\iota(t)$ and $(x,y)$ is a sink in $\Sigma_{x_{0}}(t)$.
Since $d(x,y)<\rho$, $x,y$ lie in a local cone with vertex $o$.

Since $(x,y)$ is a sink, by Proposition \ref{largeprop} $\gamma_{x},\gamma
_{y}$ coincide up to $o$ and $d(o,x)=d(o,y)=\frac{d(x,y)}{2}<t-\sqrt
{t^{2}-\frac{\rho^{2}}{4}}$. We also have
\begin{equation}
d(x_{0},o)=t-\frac{\iota(t)}{2}=\sqrt{t^{2}-\frac{\rho^{2}}{4}\text{.}}
\label{feq}%
\end{equation}
Moreover, in $S_{o}$, $\alpha(\gamma_{ox}^{\prime},\overline{\gamma_{o}%
}^{\prime})=\alpha(\gamma_{oy}^{\prime},\overline{\gamma_{o}}^{\prime}%
)=\pi\geq\frac{\pi}{2}$. By assumption there is a curve $c$ in $S_{o}$ from
$\gamma_{ox}^{\prime}$ to $\gamma_{oy}^{\prime}$ such that for all $q$,
$\alpha(\overline{\gamma_{o}}^{\prime},c(q))\geq\frac{\pi}{2}$. Since radial
geodesics do not bifurcate, and the cone radius is $\frac{\rho}{2}$, there is
a point $x^{\prime}$ on the unique (inside the cone) extension of $\gamma
_{ox}$ of distance $\frac{\rho}{2}$ from $o$, and an analogous point
$y^{\prime}$. Now the curve $\frac{\rho}{2}\cdot c$ is defined from
$x^{\prime}$ to $y^{\prime}$. By the CAT(0) inequality, for every $s$, by
Equation (\ref{feq})
\[
d(\frac{\rho}{2}\cdot c(s)),x_{0})^{2}\geq d(x_{0},o)^{2}+\frac{\rho^{2}}%
{4}=t^{2}\text{.}%
\]
Therefore the projection $\widetilde{c}$ of $\frac{\rho}{2}\cdot c$ onto
$\Sigma_{x_{0}}(t)$ is defined and joins $x$ and $y$. Moreover, since
$\widetilde{c}$ remains inside $B(o,\frac{\rho}{2})$ by the triangle
inequality it remains inside $B(x,\rho)\cap B(y,\rho)$ and we may take
$K=\rho$ to finish the proof of the first part of the theorem.

Suppose that $S_{o}$ with the angle metric is a geodesic space. Suppose that
$a,b,c\in S_{o}$ with $a,b\notin B(c,\frac{\pi}{2})$. Since $X$ is
geodesically complete and locally conical, Lemma \ref{gclem} implies that
$S_{o}$ is also geodesically complete. By the same lemma, the geodesics
$\gamma_{ca}$ and $\gamma_{cb}$ extend to geodesics of length $\pi$ to points
$a^{\prime},b^{\prime}$, respectively. Since $d(a^{\prime},b^{\prime})\leq\pi
$, any geodesic from $a^{\prime}$ to $b^{\prime}$ must remain outside
$B(c,\frac{\pi}{2})$ by the triangle inequality. The extensions of
$\gamma_{ca}$ and $\gamma_{cb}$ also remain outside $B(c,\frac{\pi}{2})$ and
therefore there is a path from $a$ to $a^{\prime}$ then $a^{\prime}$ to
$b^{\prime}$ then $b^{\prime}$ to $b$ that stays outside $B(c,\frac{\pi}{2})$.

Now suppose $S_{o}$ has no cut points and is locally conical with all cone
radii at least $\frac{\pi}{2}$. Suppose $a,b,c$ are as in the previous
paragraph. Since $c$ is not a cut point, the complement of $c$ is connected,
hence path connected. That is, there is a curve from $a$ to $b$ that misses
$c$. Now any segment of $c$ that enters $B(c,\frac{\pi}{2})$ can be homotoped
onto $\Sigma_{c}(\frac{\pi}{2})$ using the radial retraction (Remark
\ref{radial}), resulting in a curve from $a$ to $b$ that stays outside
$B(c,\frac{\pi}{2})$.
\end{proof}

\begin{acknowledgement}
In connection with this paper I had useful conversations with Ross Geoghegan,
Mike Mihalik, and Kim Ruane.
\end{acknowledgement}

\end{document}